\documentclass[12pt]{article}

\usepackage{amssymb,amsmath,amsfonts,amssymb}
\usepackage{graphics,graphicx,color}
\usepackage{empheq}

\def\@abssec#1{\vspace{.05in}\footnotesize \parindent .2in
{\bf #1. }\ignorespaces}

\graphicspath{{/EPSF/}{Figures/}{images/}}
\DeclareGraphicsExtensions{.eps}

\newtheorem{theorem}{Theorem}[section]
\newtheorem{lemma}[theorem]{Lemma}

\newtheorem{corollary}[theorem]{Corollary}

\def \Rm {\mathbb R}

\def \Cm {\mathbb C}
\def \Zm {\mathbb Z}

\newcommand{\eps}{\varepsilon}

\newcommand{\mF}{\mathcal F}

\newcommand{\mD}{\mathfrak D}
\newcommand{\mJ}{{\mathfrak J}}
\newcommand{\fa}{{\mathfrak a}}

\newcommand{\sE}{\mathsf E}
\newcommand{\mE}{\mathcal E}

\newcommand{\cout}[1]{}

\newcommand{\sgn}[1]{\,{\rm sign}(#1)}

 \renewcommand{\arraystretch}{1.5}

\title{Topological Equatorial Waves and Violation (or not) of the Bulk Edge Correspondence}

\author{Guillaume Bal \thanks{Departments of Statistics and Mathematics and Committee on Computational and Applied Mathematics, University of Chicago, Chicago, IL 60637; guillaumebal@uchicago.edu} \and Jiming Yu \thanks{Committee on Computational and Applied Mathematics, University of Chicago, Chicago, IL 60637; tommenix@uchicago.edu}}

\begin{document}
 
\maketitle



\begin{abstract}
  Atmospheric and oceanic mass transport near the equator display a well-studied asymmetry characterized by two modes moving eastward. This asymmetric edge transport is characteristic of interfaces separating two-dimensional topological insulators. The northern and southern hemispheres are insulating because of the presence of a Coriolis force parameter that vanishes only in the vicinity of the equator. A central tenet of topological insulators, the bulk edge correspondence, relates the quantized edge asymmetry to bulk properties of the insulating phases, which makes it independent of the Coriolis force profile near the equator. We show that for a natural differential Hamiltonian model of the atmospheric and oceanic transport, the bulk-edge correspondence does not always apply. In fact, an arbitrary quantized asymmetry can be obtained for specific, discontinuous,  such profiles. The results are based on a careful analysis of the spectral flow of the branches of absolutely continuous spectrum of a shallow-water Hamiltonian. Numerical simulations validate our theoretical findings.
\end{abstract}
 

\renewcommand{\thefootnote}{\fnsymbol{footnote}}
\renewcommand{\thefootnote}{\arabic{footnote}}

\renewcommand{\arraystretch}{1.1}





%

\section{Introduction} \label{sec:intro}

Robust asymmetric transport along interfaces separating two-dimensional insulating bulks has been observed in many areas of applied science such as solid state physics, photonics, and geophysics \cite{avron1983homotopy,bernevig2013topological, delplace2017topological, prodan2016bulk, sato2017topological, volovik2009universe, witten2016three}.  That atmospheric (and oceanic) mass transport near the equator displays such an asymmetry and a robustness to perturbation was recently given a topological interpretation in \cite{delplace2017topological}; see also \cite{bal2022topological,bal2023topological,graf2021topology,jud2024classifying,quinn2024approximations,rossi2024topology,souslov2019topological,tauber2019bulk,tauber2023topology,zhu2023topology} and their references for subsequent analyses on this and related problems.

\paragraph{Topological shallow water model.}
The simplest model of atmospheric transport, and the only one we consider in this paper, is the following linearized system of shallow water equations
\begin{equation}\label{eq:H}
  i\partial_t \psi(t,x,y) = H\psi (t,x,y), \qquad H = \begin{pmatrix} 0 & D_x & D_y \\ D_x & 0 & if(y) \\ D_y & -if(y) & 0\end{pmatrix},\quad \psi =  \begin{pmatrix} \eta  \\ u \\ v \end{pmatrix},
\end{equation}
where $t$ is time, $x$ the spatial coordinate along the equator, $y$ the spatial coordinate across the equator in an appropriate linearization so that $(x,y)\in\Rm^2$ the Euclidean plane, $D_y=\frac 1i\partial_x$ and $D_y=\frac 1i\partial_y$ while $f(y)$ is a Coriolis force parameter. Mass transport is modeled by $\eta(t,x,y)$ the atmosphere height, $u(t,x,y)$ its horizontal velocity, and $v(t,x,y)$ its vertical velocity. See \cite{delplace2017topological} for a derivation of this model from Boussinesq primitive equations.

A simple solution to the above system of equations is $v(t,x,y)=0$ with $\eta(t,x,y)=u(t,x,y)=e^{-F(y)}\phi(t-x)$ for $F(y)=\int_0^y f(z)dz$ and $\phi(x)$ an arbitrary function. This solution corresponds to the Kelvin mode \cite{delplace2017topological,matsuno1966quasi} and does propagate eastward and dispersion-less with normalized (group) velocity $1$ along the equator.

The above operator $H$ is Hermitian and self-adjoint with an appropriate domain of definition in $L^2(\Rm^2;\Cm^3)$ \cite{bal2022topological,rossi2024topology}. The observed asymmetric transport along the equator may be explained by the spectral properties of $H$.

When the Coriolis force parameter $f$ is constant, the operator $H\equiv H_B$ may be diagonalized (fibered) as 
\begin{equation}\label{eq:hatHB}
  H_B = \mF^{-1} \hat H_B \mF,\qquad  (\xi,\zeta)\to \hat H_B(\xi,\zeta) = \begin{pmatrix} 0 & \xi& \zeta \\ \xi & 0 & if \\ \zeta & -if& 0\end{pmatrix},
\end{equation}
where $\mF=\mF_{(x,y)\to(\xi,\zeta)}$ represents Fourier transformation with $(\xi,\zeta)$ dual variables to $(x,y)$ with the convention that $\mF f(\xi,\zeta)=\int_{\Rm^2} f(x,y) e^{-i(x\xi+y\zeta)} dxdy$. When other conventions are used, the Coriolis force parameter $f$ may then appear as $-f$. The above Hamiltonian may be written as $\xi\gamma_1+\zeta\gamma_4-if\gamma_7$ for $\gamma_i$ an appropriate set of Gell-Mann matrices. Any set of such $3\times3$ generators of the algebra ${\mathfrak{su}}(2)$ may be used instead. See \cite{bal2022topological,delplace2017topological,zhu2023topology} for equivalent such conventions.

We then observe that $H$ admits three branches of absolutely continuous spectrum parametrized by
\[
  (\xi,\zeta) \mapsto  E_0(\xi,\zeta)=0,\qquad (\xi,\zeta) \mapsto  E_\pm(\xi,\zeta)=\pm \sqrt{\xi^2+\zeta^2+f^2}.
\]
When $f\not=0$, we thus observe  the presence of two spectral gaps in $(-|f|,0)$ and $(0,|f|)$. The Coriolis force parameter is positive in the northern hemisphere and negative in the southern hemisphere. A reasonable model is in fact given by $f(y)=\beta y$ in a $\beta-$plane model \cite{delplace2017topological,matsuno1966quasi}.

A large positive value $f_+$ of the Coriolis parameter in the northern hemisphere for $y\gtrsim1$ combined with a large negative value $f_-$ in the southern hemisphere $y\lesssim-1$ therefore leads to two insulating half-spaces joined along the equator for any excitation with frequency residing in the gaps $(-|f_+|\wedge |f_-|,0)$ or  $(0, |f_+|\wedge |f_-|)$. 

It turns out that we can associate topological invariants to these insulating bulks. When the Coriolis parameters in the two half-space bulks have different signs, this topological invariant takes a value equal to $2$; see  \cite{bal2022topological,delplace2017topological,tauber2019bulk}. 

An application of a general principle, the {\em bulk-edge correspondence}, then states that an edge transport asymmetry compensates for this bulk imbalance, resulting in the observed asymmetric transport along the equator captured by two well known eastward propagating modes, the Kelvin mode and the Yanai mode \cite{delplace2017topological,matsuno1966quasi}. The topological nature of the asymmetry ensures its robustness to perturbations; see however \cite{bal2022topological,quinn2024approximations} for restrictions to this statement. 

We refer to, e.g., \cite{bernevig2013topological,fukui2012bulk,RevModPhys.82.3045,hatsugai1993chern,jezequel2023mode,kane2013topological} and \cite{bal2022topological,bal2023topological,bourne2018chern,drouot2021microlocal,elbau2002equality,prodan2016bulk,quinn2024approximations,schulz2000simultaneous} for references in the physical and mathematical literatures on this ubiquitous phenomenon, which finds applications in many different contexts but still defies any simple derivation. 

\paragraph{Edge current observable.}

While many different invariants may be defined, the invariant most closely related to a physical observable is based on the following edge current observable \cite{bal2022topological,bal2023topological,bourne2018chern,drouot2021microlocal,elbau2002equality,prodan2016bulk,schulz2000simultaneous}
\begin{equation}\label{eq:sigmaI}
  \sigma_I= {\rm Tr}\, i [H,P] \varphi'(H).
\end{equation}
Here $P(x):\Rm\to[0,1]$ is a  smooth function equal to $P(x)=0$ for $x<x_0-\eta$ and equal to $P(x)=1$ for $x>x_0+\eta$ for some $x_0\in\Rm$ and $\eta>0$. The operator $i[H,P]$ may then intuitively be interpreted as a current operator for atmospheric mass moving (per unit time) from the left to the right of a vertical line $x=x_0$.

The density of states $\varphi'(H)$ is modeled by a function $\varphi'(h)\geq0$ integrating to $1$ and supported in a frequency window $[E_0,E_1]$ inside each gap of the bulk Hamiltonians. When $f(y)$ has unbounded range, as we will assume in this paper to slightly simplify derivations, then the support of $\varphi'(H)$ is in fact arbitrary. 

It is possible to show that $i [H,P] \varphi'(H)$ is trace-class and hence $\sigma_I$ is a well-defined real number \cite{bal2022topological,quinn2024approximations} with ${\rm Tr}$ denoting operator trace on $L^2(\Rm^2;\Cm^3)$.

For Hamiltonians that are invariant with respect to translations along the edge, the edge current can  be calculated by means of spectral flows of branches of absolutely continuous spectrum as follows.  Let $H$ be as defined as in \eqref{eq:H} and $\mF_{x\to\xi}$ the Fourier transform in the first variable only. Then, we have the partial diagonalization (fibration)
\begin{equation}\label{eq:hatH}
  H = \mF_{\xi\to x}^{-1} \hat H \mF_{x\to\xi} ,\qquad \Rm\ni \xi \to \hat H(\xi) =  \begin{pmatrix} 0 & \xi& D_y \\ \xi & 0 & if(y) \\ D_y& -if(y)& 0\end{pmatrix}.
\end{equation}
It is convenient to decompose $L^2(\Rm^2)$ as a direct sum of spaces $L^2_\xi(\Rm)$ of functions of the form $e^{ix\xi}u(y)$ and then realize that $\hat H(\xi)$ is a family (parametrized by $\xi$) of (unbounded) self-adjoint operators on $L^2_\xi(\Rm;\Cm^3)$ \cite{bal2022topological,rossi2024topology}. The operator $\hat H(\xi)$ then admits only discrete spectrum away from $E=0$ when the range of $f(y)$ is unbounded; see  \cite{bal2022topological,teschl2009mathematical} and Lemma \ref{lem:discrete} below. Since we are interested in a range of frequency inside the bulk spectral gaps and values of $f(y)$ for $|y|\gg1$ are irrelevant to describe modes confined close to the equator $y\approx 0$, we will assume that $f(y)$ has unbounded range and as a result from Lemma \ref{lem:discrete} obtain that $\hat H(\xi)$ admits countable discrete spectrum $E_n(\xi)$ for $n\in\Zm$ away from frequency $E=0$. 

Consider the branches $E_n(\xi)\geq0$. The branches $\xi\mapsto E_n(\xi)$ are real analytic away from $E=0$ since they are one-dimensional and $\xi\mapsto\hat H(\xi)$ is analytic \cite[Chapter VII.1.1]{Kato66}; see also Lemma \ref{lem:discrete}. We will obtain that as $|\xi|\to\infty$, then $E_n(\xi)$ either converges to a constant or diverges to $\pm\infty$.  Assuming this asymptotic behavior, we may then associate to each branch for which $E_n(\xi)\not=0$ a spectral flow for any $\Rm\ni\alpha>0$:
\begin{equation}\label{eq:SF}
   {\rm SF}(E_n;\alpha) = \frac12 \Big( \lim_{\xi\to\infty} \sgn{E_n(\xi)-\alpha} -  \lim_{\xi\to-\infty} \sgn{E_n(\xi)-\alpha}\Big).
\end{equation}
For branches defined on $(-\infty,\xi_0]$ such that $E_n(\xi_0)=0$, we then define
\begin{equation}\label{eq:SF2}
   {\rm SF}(E_n;\alpha) =  -\frac12 \Big(1+  \lim_{\xi\to-\infty} \sgn{E_n(\xi)-\alpha}\Big).
\end{equation}
For branches defined on $[\xi_0,\infty)$ such that $E_n(\xi_0)=0$, we similarly define
\begin{equation}\label{eq:SF3}
   {\rm SF}(E_n;\alpha) = \frac12 \Big(1+   \lim_{\xi\to+\infty} \sgn{E_n(\xi)-\alpha}\Big).
\end{equation}


When $f'$ vanishes at most on a discrete set, then we will see that $\xi_0$ above has to equal $0$. The spectral flow of the $n$th branch then takes values in $\{-1,0,1\}$ depending on how the branch $E_n(\xi)$ crosses a given frequency level $\alpha$. See Figure \ref{fig:1} below for illustrations of spectral branches. Spectral flows for branches $E_n(\xi)\leq0$ are defined similarly.

Further assume that $\varphi'(h)$ is supported in an interval $[E_0,E_1]$ such that all above spectral flows are independent of $\alpha\in [E_0,E_1]$. Then the edge current observable \eqref{eq:sigmaI} is related to the spectral flow of $H$ as follows:
\begin{equation}\label{eq:edgeSF}
  2\pi \sigma_I = \sum_n {\rm SF}(E_n;\alpha).
\end{equation}
In such a sum, we assume that the number of branches crossing the level $\alpha$ is finite. We will obtain that this hypothesis holds for the branches of \eqref{eq:H} except possibly for a well-identified finite set. 
The derivation of \eqref{eq:edgeSF} may be found in \cite[Theorem 3.1]{quinn2024asymmetric}; see also \cite{bal2022topological,fukui2012bulk} for the computation of edge currents by spectral flows. 

This result shows that the asymmetric edge current is entirely described by the absolutely continuous spectrum of the Hamiltonian $H$.

\paragraph{Bulk-edge correspondence (BEC).}

The above computation by spectral flow is quite general. However, it requires detailed spectral information on the Hamiltonian $H$. A general principle, the bulk-edge correspondence, relates the quantized edge current, naturally associated to a finite number of propagating edge modes, to the topological properties of the insulating bulks. 

For macroscopic continuous Hamiltonians (as opposed to discrete Hamiltonians \cite{elbau2002equality,prodan2016bulk,schulz2000simultaneous} but also to continuous operators with microscopic periodic structure \cite{drouot2021microlocal,drouot2020edge}), a first difficulty arises in that topological phases may not be defined for each (north and south)  bulk independently \cite{B19b,bal2022topological,silveirinha2015chern}. Rather, it is generally possible to introduce {\em bulk-difference} invariants $I_{BD}\in \Zm$ as defined in \cite{bal2022topological}; see also \cite{rossi2024topology}. These invariants may be interpreted as the computation of a Chern number associated to a vector bundle on top of the unit sphere. Such invariants only depend on a topological phase difference between the two bulks and do not require absolute bulk topologies, which cannot be defined for such objects without appropriate regularization \cite{B19b,bal2022topological,rossi2024topology,silveirinha2015chern}.

The bulk-edge correspondence then takes the form
\begin{equation}\label{eq:BEC}
  2\pi \sigma_I = I_{BD}.
\end{equation}
This correspondence is important in practice because the right-hand-side is typically significantly simpler to compute than the left-hand-side or the spectral flows in \eqref{eq:edgeSF}. Such a correspondence was shown under simplifying assumptions on $f(y)$ for the shallow-water problem in \cite{bal2022topological,bal2023topological}; see also \cite{quinn2024approximations} for several numerical simulations. We further comment on these results at the end of section \ref{sec:main}.

While extremely robust and proved in a variety of contexts \cite{bal2022topological,drouot2021microlocal,elbau2002equality,prodan2016bulk,schulz2000simultaneous}, the derivation of \eqref{eq:BEC} requires some hypotheses on $H$ beyond the existence of $I_{BD}$. For continuous macroscopic operators of the form of interest in this paper, a natural notion is that of {\em ellipticity} as defined in \cite{bal2022topological,bal2023topological,quinn2024approximations}. It turns out that the shallow water problem {\em does not} satisfy such ellipticity conditions, which we will not be presenting in detail here. Relevant to our discussions is the fact that the  eigenvalues of the phase-space symbol of the Hamiltonian are given by
\[
  E_0(x,y,\xi,\zeta)=0,\quad E_\pm(x,y,\xi,\zeta) = \pm \sqrt{\xi^2+\zeta^2+f^2(y)}.
\]
While the branches $E_\pm$ satisfy appropriate ellipticity conditions in the sense that they grow at least linearly as $|(\xi,\zeta)|\to\infty$, the branch $E_0$ is clearly not elliptic in that sense. This fact is responsible for a number of obstructions to the bulk-edge correspondence, both in the Euclidean-plane setting \cite{bal2022topological,quinn2024approximations,rossi2024topology} and in the half-space setting for a large class of boundary conditions \cite{graf2021topology,jud2024classifying,tauber2019anomalous,tauber2023topology}.

\medskip 

It is possible to restore the validity of the BEC by perturbing the definition of the edge invariant in \eqref{eq:sigmaI}.  In \cite{rossi2024topology}, the dispersion relation $\xi\mapsto E(\xi)$ is modified by an appropriate change of variables that restores the BEC. This may equivalently seen as modifying the straight level sets $E(\xi)=\alpha$ by replacing $\alpha$ by $\alpha(\xi)$ that tends to $\infty$ as $|\xi|\to\infty$. The invariant is, however, no longer related to the physical observable $\sigma_I$. 

A general regularization procedure is proposed in \cite{bal2023topological,quinn2024approximations}. It is based on the fact that the flat band $E_0(\xi)$ fails to be elliptic but is topologically trivial (any reasonable associated invariant vanishes). In these references, the Hamiltonian is perturbed in such a way that the flat band $E_0=0$ is replaced by a band of the form $E_0=\mu\sqrt{1+\xi^2+\zeta^2+f^2}$ while the corresponding generalized eigenvector and projectors are not modified. This naturally defines a modified Hamiltonian $H_\mu$. For any $\mu\not=0$, the problem thus becomes elliptic and the general BEC theory of \cite{bal2022topological,bal2023topological,quinn2024approximations} applies so that \eqref{eq:BEC} holds. Since the flat band $E_0=0$ when $\mu=0$ is topologically trivial \cite{bal2022topological,graf2021topology,rossi2024topology}, the choice of sign $\mu>0$ or $\mu<0$ is irrelevant as far as \eqref{eq:BEC} is concerned. Indeed, the band topology only depends on the generalized eigenvectors and corresponding projectors, which are independent of $\mu$ by construction, and not on the spectral bands themselves \cite{bernevig2013topological}.

While such regularizations `solve' the issue of the BEC, they come with a number of caveats. Indeed, the branches of absolutely continuous spectrum are modified by the regularization $|\mu|\ll1$ for very large values of the wavenumbers $\xi$. At such values, the Hamiltonian model is likely to be inaccurate as high wavenumbers are typically absent in the system because of strong dissipation effects not captured by Hamiltonian dynamics. Moreover, theoretical derivations and numerical simulations in \cite{quinn2024approximations} show that $2\pi\sigma_I[H+V]$ is not stable for some compact perturbations $V$ when $\mu=0$, whereas the theory of \cite{bal2022topological,bal2023topological,quinn2024approximations} shows that $2\pi\sigma_I[H+V]=2\pi\sigma_I[H]$ for any compactly supported perturbation $V$ as soon as $V\not=0$. 

An alternative restoration of the bulk edge correspondence has recently been proposed in \cite{onuki2023bulk}. It is based on extending the $\beta-$plane model \eqref{eq:H} to a $4\times4$ system of equations. This model restores a bounded spectrum for which a standard bulk edge correspondence then applies.

\paragraph{Validity or not of the Bulk-edge correspondence.}

In this paper, we consider the un-regularized, un-modified, $3\times3$ Hamiltonian $H$ defined in \eqref{eq:H} and aim to compute its spectral flow for general profiles $f(y)$. When $f(y)=y$ or $f(y)= f_+ \chi_{y>y_0} + f_-\chi_{y<y_0}$ takes two values (for instance $f(y)=f_0\sgn{y}$), then the discrete spectrum of $\hat H(\xi)$ may be computed explicitly \cite{delplace2017topological,matsuno1966quasi}, \cite[Appendix]{bal2022topological}; see also Fig. \ref{fig:1}. For more general profiles $f(y)$, it seems difficult to derive explicit expressions for the diagonalization of $H$. However, we may obtain sufficient information to compute its spectral flow, which is sufficient in regard of \eqref{eq:edgeSF}. 

\begin{figure}[htbp]
  \begin{center}
  \includegraphics[width=7cm]{m=2001,f=y.jpg}  \hspace{1cm} \includegraphics[width=7cm]{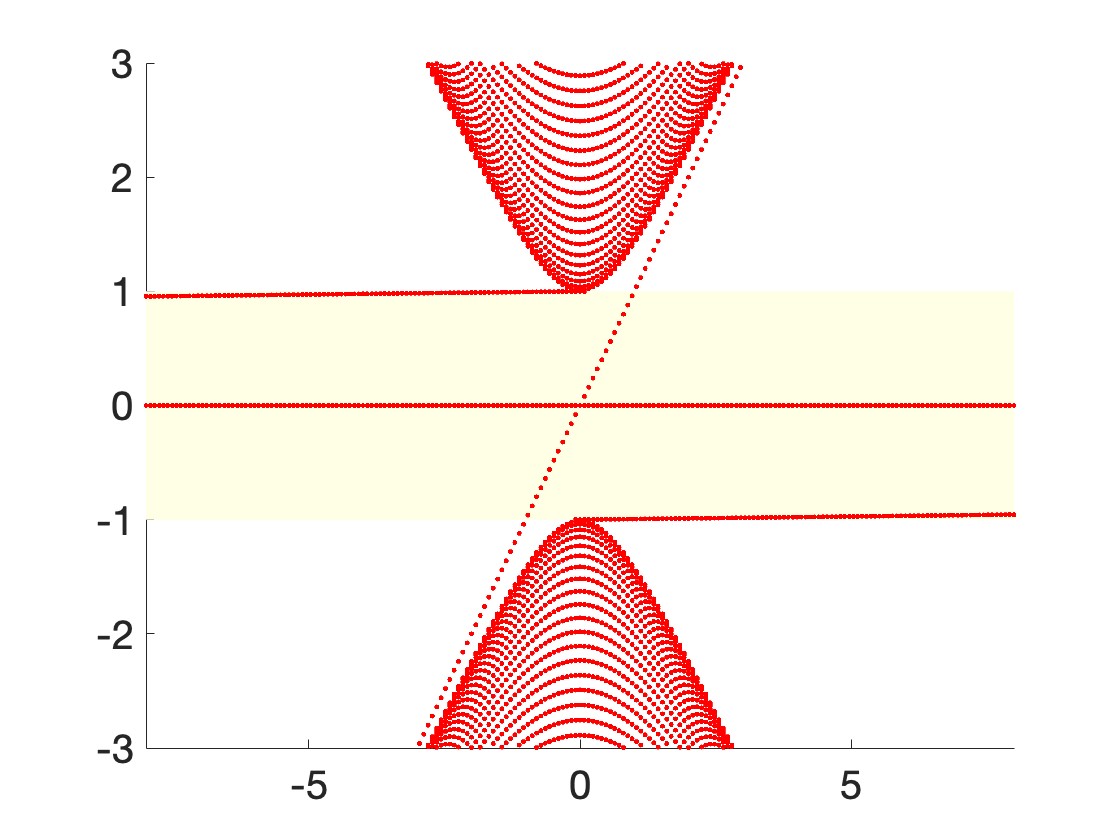}
  \end{center}
  \caption{Left: spectrum of $H$ when $f(y)=y$ with a spectral flow equal to $2$ as dictated by the BEC. Right: spectrum of $H$ when $f(y)=\sgn{y}$. For $\varphi'(H)$ supported in $(0,1)$, the spectral flow equals $1$: only the Kelvin mode contributes to the spectral flow while the Yanai mode is pushed away from the frequency range $(0,1)$ by the jump of $f(y)$. These branches were obtained by the numerical algorithm presented in section \ref{sec:num}.} 
  \label{fig:1}
\end{figure}

Our main objective in this paper is to show that for sufficiently smooth profiles $f(y)$, then the BEC holds while for $f(y)$ with an appropriate number of discontinuities, the violation of the BEC can be arbitrary. 

\paragraph{Outline.} The rest of the paper is structured as follows. Our main assumptions on the profile $f(y)$, our main results in Theorems \ref{thm:smooth} and \ref{thm:SF}, as well as a discussion on their validity and interpretation, are presented in section \ref{sec:main}. The proofs of the main results are postponed to section \ref{sec:results}. A algorithm to estimate the branches $E_n(\xi)$ and their spectral flow numerically is presented in section \ref{sec:num}. This provides a numerical validation of the theoretical analyses.

\section{Main results}\label{sec:main}
We recall that the main operator of interest is $H$ defined in \eqref{eq:H} and that it may be decomposed (fibered) as $H=\mF_{\xi\to x}^{-1}\hat H(\xi) \mF_{x\to\xi}$ with $\hat H(\xi)$ defined in \eqref{eq:hatH}. 

For the rest of the section, we assume that $f(y)$ is such that $f(y) \sgn{y}\to+\infty$ as $y\to\pm\infty$. We also assume the existence of a finite number of (distinct, increasing) points $y_j$ for $1\leq j\leq J$ such that $f(y)$ is uniformly of class $C^1$ away from the points $y_j$, where $[f](y_j)= f(y_j^+)-f(y_j^-)\not=0$ is defined and finite. Here $f(y^\pm)=\lim_{0<\eps\to0} f(y\pm\eps)$. To simplify derivations, we make the technical (and unnecessary) assumption that for some $\eta>0$, we have $f(y) = f(y_j^+) \chi_{y_j^+<y<y_j^++2\eta} + f(y_j^-) \chi_{y_j^--2\eta<y<y_j^-}$ for $y\in (y_j-2\eta,y_j+2\eta)$, i.e., the Coriolis parameter $f(y)$ is piecewise constant near jumps.

The first main result of this paper is that the BEC holds when $f(y)$ is continuous. 
\begin{theorem}[BEC]\label{thm:smooth}
  Assume further that $f'(y)$ is bounded and that $\varphi'(E)$ is supported away from $E=0$. Then the spectral flow of $H$ is given by \eqref{eq:edgeSF} with $2\pi\sigma_I=2$.
\end{theorem}
The value of $2$ above is indeed that predicted by the invariant $I_{BD}$ in \eqref{eq:BEC}, whose computation is not entirely trivial; see \cite{bal2022topological,delplace2017topological,rossi2024topology,tauber2019bulk} for different methods to compute this bulk-difference invariant or difference of bulk invariants after appropriate regularization.

\medskip

To compute the spectral flow of $H$ in the presence of jumps of $f$, we first need to define the values
\[
   f_{jo}=\frac{f(y_j^+)-f(y_j^-)}2,\qquad f_{je}=\frac{f(y_j^+)+f(y_j^-)}2 ,
 \]
and the sets of half-jump values
\[
  \mE_L= \Big\{ f_{jo} \ ; \  1\leq j\leq J \ \& \  f_{jo}>0\Big\},\quad \mE_R= \Big\{-f_{jo}\  ; \  1\leq j\leq J \ \& \  f_{jo}<0\Big\}.
\]
For $E>0$, we also define $\mJ_L(E)$ as number of indices $1\leq j\leq J$ such that there is $E_j\in \mE_L$ such that $E_j>E$. We similarly define  $\mJ_R(E)$ as number of indices $1\leq j\leq J$ such that there is $E_j\in \mE_R$ such that $E_j > E$. We will observe that the spectrum of $H$ is symmetric under $(\xi,E)\to (-\xi,-E)$ and hence only treat the case $E>0$ since the spectral flow is invariant under $\alpha\to -\alpha$ in \eqref{eq:edgeSF}.

In other words, $\mJ_L(E)$ is the number of positive jumps with half-jump values above $E$. Similarly, $\mJ_R(E)$ is the number of negative jumps with half-jump (absolute) values above $E$.

We then have the result:
\begin{theorem}[BEC violation]\label{thm:SF}
  Let  $E>0$ with $E\not\in \mE_L\cup \mE_R$. Assume $\varphi$ so that  $E\in {\rm supp}\, \varphi'$ while ${\rm supp}\,\varphi' \cap (\mE_L\cup \mE_R\cup\{0\})=\emptyset$. Then we have
  \[
    2\pi \sigma_I  = 2 - \mJ_L(E)+ \mJ_R(E).
  \]
\end{theorem}
We thus deduce a violation of the BEC as soon as $ \mJ_L(E)\not=\mJ_R(E)$, i.e., as soon as for a frequency level $\alpha=E>0$, then the number of positive jumps greater than $2E$ is not equal to the number of negative jumps greater than $2E$ in absolute value.

When $f(y)$ has a unique jump of size $+2$ as does $f(y)=\sgn{y}$ (extended by a function $f(y)\to\pm\infty$ as $y\to\pm\infty$), then we obtain that 
\[
  2\pi \sigma_I = 2-\mJ_L(\alpha)=1,\quad 0<\alpha<1,\qquad 2\pi \sigma_I = 2-\mJ_L(\alpha)=2,\quad 1<\alpha.
\]
When $f(y)$ has (exactly) $J$ separate jumps of size $+2$, then we obtain more generally that 
\[
  2\pi \sigma_I = 2-\mJ_L(\alpha)=2-J,\quad 0<\alpha<1,\qquad 2\pi \sigma_I = 2-\mJ_L(\alpha)=2,\quad 1<\alpha.
\]
When $f(y)$ has $J$ separate jumps of size $-2$ instead, then we find that 
\[
  2\pi \sigma_I = 2+ \mJ_R(\alpha)=2+J,\quad 0<\alpha<1,\qquad 2\pi \sigma_I = 2+ \mJ_R(\alpha)=2,\quad 1<\alpha.
\]
These examples provide profiles of $f(y)$ for which the violation of the BEC can take any integer value. Several additional illustrations are provided in section \ref{sec:num}.

\medskip

We may thus interpret the simple shallow-water model \eqref{eq:H} as either satisfying the BEC when $f(y)$ is smooth or violating it when $f(y)$ displays discontinuities. It is not clear why the Coriolis force parameter should have any discontinuities so that the results of Theorem \ref{thm:smooth} may seem appropriate practically.

As observed in many other contexts \cite{bal2022topological,bal2023topological,graf2021topology,jud2024classifying,quinn2024approximations,rossi2024topology,souslov2019topological,tauber2019bulk,tauber2023topology,zhu2023topology}, the BEC of models of the form \eqref{eq:H} is fragile. Theorem \ref{thm:SF} identifies this fragility explicitly without the introduction of boundary conditions. For the latter problems, we refer to \cite{graf2021topology,jud2024classifying,tauber2019bulk,tauber2023topology}. 

Note that in the smooth case, the above result is a `verification' that the bulk edge correspondence holds, namely that `$2=2$', not a `derivation' as in different but general scenarios such as in, e.g.,
\cite{bal2022topological,bal2023topological,bourne2018chern,drouot2021microlocal,elbau2002equality,prodan2016bulk,quinn2024approximations,schulz2000simultaneous}. A derivation of the BEC that applies to a general class of operators including \eqref{eq:H} and to appropriate perturbed operator $H+V$ has been obtained in \cite[Proposition 4.9]{bal2022topological} when the derivative $f'(y)$ is bounded and {\em sufficiently small} as an application of a G\aa rding inequality.

This fragility is also not only a matter of regularization. While the above theorems describe the spectral flow of an unperturbed operator $H$ in \eqref{eq:H}, and while \cite[Proposition 4.9]{bal2022topological} shows the stability of the edge current $2\pi\sigma_I$ for a perturbed operator $H+V$ for certain perturbations $V$, that same result indicates that $2\pi\sigma_I$ should  {\em not} be stable for all $3\times3$ perturbations $V$. This was confirmed in numerical simulations in \cite[Eq.(4.2) \& Fig. 4]{quinn2024approximations}. For even quite small perturbations of the form $V={\rm Diag}(0,V_{22},V_{33})$, numerical simulations of $2\pi\sigma_I[H+tV]$ show that the edge current quickly destabilizes as $t$ increases. This instability should be reasonably independent of any regularization or modification of the model whether they restore the BEC or not. In contrast, $2\pi\sigma_I[H+V]$ is very robust against perturbations of the form $V={\rm Diag}(V_{11},0,0)$ or of the form $V_1\gamma_1+v_4\gamma_4+V_7\gamma_7$ \cite[Eq.(4.2) \& Fig. 4]{quinn2024approximations} in agreement with the prediction in \cite[Proposition 4.9]{bal2022topological}. The observed stability of the Kelvin and Yanai modes in practice may indicate the presence of symmetries that prevent destabilizing perturbations $V$ from occurring.

\section{Derivation of the main results}\label{sec:results}
This section presents a proof of the results of Theorems \ref{thm:smooth} and \ref{thm:SF}.

\paragraph{Hamiltonian.} We start from the Hamiltonian $H$ defined in \eqref{eq:H} and recall that after partial Fourier transform, we have
\begin{equation}\label{eq:hatHmE}
 \hat H -E = \left(\begin{matrix}  -E & \xi & D_y \\  \xi & -E & if \\ D_y & -if &  -E \end{matrix} \right).
\end{equation}
Since $f(y)\to \pm\infty$ as $y\to\pm\infty$, we deduce from Lemma \ref{lem:discrete} below that the self-adjoint operator $\hat H(\xi)$ has purely discrete spectrum away from $E=0$, where it may have essential (point) spectrum. This defines real-analytic branches $\xi\mapsto E_n(\xi)$ of absolutely continuous spectrum of the operator $H$.  The proof of the above theorems is based on computing the spectral flow \eqref{eq:SF} of each branch for (almost) all frequency levels $\alpha$.

We observe that for $\Gamma={\rm Diag}(1,1,-1)$, then
\[
 -\Gamma (\hat H(\xi)-E) \Gamma = \left(\begin{matrix}  E & -\xi & D_y \\  -\xi & E & if \\ D_y & -if &  E \end{matrix} \right).
\]
In other words $(\hat H(\xi)-E)\psi=0$ is equivalent to $(\hat H(-\xi)+E)\Gamma\psi=0$. This implies that $E_n(-\xi)=-E_{n'}(\xi)$ and that the continuous spectrum of $H$ is symmetric under $(\xi,E)\to(-\xi,-E)$. As a consequence, we focus on the branches $E_n(\xi)>0$ since the branches $0>E_n(\xi)$ are obtained by symmetry and have the same spectral flow properties.


We define $\fa:=\partial_y+f(y)$ so that $D_y-if=-i\fa$ and $D_y+if=i \fa^*$. Since we assume that $f(y)\to\pm\infty$ as $y\to\pm\infty$, the domains of definition of $\fa$ and $\fa^*$ are 
\[
 \mD(\fa) = \mD(\fa^*) = \{ u \in L^2(\Rm); \ f u \in L^2(\Rm), \ u' \in L^2(\Rm)\}.
\] 
On such a domain, $\fa^*$ is invertible and bounded as an operator from $\mD(\fa^*)$ to $L^2$. The operator $\fa$ has a non-trivial kernel given by a normalization of the function $e^{-F(y)}$ where $F(y)=\int_0^y f(z)dz$. We verify as in, e.g., \cite{bal2022topological}, that the kernel of $\fa^*$ is trivial and hence that the index (dimension of kernel minus dimension of co-kernel) of the Fredholm operator $\fa$ (from $\mD(\fa)$ to $L^2$) is $1$. 


 
 We now deduce several properties of the branches $\xi\mapsto E_n(\xi)$ directly from the equation $(\hat H-E)\psi=0$ using the expression \eqref{eq:hatHmE}. We recall the notation $\psi=(\eta,u,v)^t$.

\paragraph{Case $E(\xi) = \xi$.} This is the Kelvin mode. It is always present with $\eta = u = ce^{-F(y)} $ and $v = 0$. Conversely, $E = \xi$ implies that $\fa^*v = 0$ and hence $v = 0$ while $\eta = u$ solves $\fa u = 0$ and hence is also unique $\eta = u = ce^{-F(y)} $. No other branch $E_n(\xi)$ is thus allowed to cross $E_0(\xi) = \xi$.

\paragraph{Case $v=0$.} This implies $E^2=\xi^2$ for otherwise $u=\eta=0$. The case  $E=-\xi$ implies $u=-\eta$ and then $\fa^*u=0$ and hence $u=0$. So, a third component $v=0$ is only possible when $E=\xi$ as above.  

\paragraph{Case $E(\xi)=-\xi$.}  This constraint implies that $v$ is fixed as a solution of $\fa v=0$, i.e., $v(y)=c e^{-F(y)}$ for some $c\not=0$. For $\xi<0$ since $E>0$, we then obtain
\[ -\fa^*\eta= i\frac{f^2-\xi^2}\xi v\]
has a unique well defined solution $\eta$ with source in $L^2$ since $f^2v$ is clearly in $L^2(\Rm)$. Then $u+\eta=-\frac{ifv}\xi$ implies that $\eta$ is uniquely characterized. Now,
\[ 0 = (\eta,\fa v) = (\fa^*\eta,v)   =  i [\xi\|v\|^2 - \xi^{-1} \|fv\|^2].\] 
Since $\xi<0$, this implies the unique solution 
\begin{equation}\label{eq:Ecrossing}
  \xi = - \frac{\|fv\|}{\|v\| }= - \frac{\|v'\|}{\|v\|}  = -\dfrac{\| f e^{-F} \| }{\| e^{-F}\|} = -E.
\end{equation}
This provides a unique solution for $E(\xi)=-\xi$ and $E>0$.  In other words, there is a unique branch crossing the half-line $0<E=-\xi$. This branch models the Yanai mode.

\paragraph{Branches are simple away from $E=0$.} We now show that each branch $\xi\mapsto E(\xi)$ is simple. 
We already know this to be the case for $E(\xi)=\xi$ where $v=0$. Otherwise, $v\not=0$ and after elimination of $(\eta,u)$  as in \eqref{eq:etau}-\eqref{eq:v} below solves:
\[ 
(-\partial^2_y + f^2 + \frac \xi E f') v= (E^2-\xi^2)v.
\]
We use a standard method of Wronskians for the above Sturm-Liouville problem as in, e.g., \cite{teschl2009mathematical}.
Assume $f$ smooth except at a finite set of distinct points $y_j$. Start on $(-\infty,y_1)$ on which $f$ is smooth. Assume two normalized solutions $v_{1,2}$. Then the Wronskian $v_1v_2'-v_2v_1'$ is constant and hence vanishing on $(-\infty,y_1)$ since the functions $v$ are normalized. At points of singularities of $f$, we have for $k=1,2$ that 
\[ -[v_k'](y_j) + \frac \xi E [f](y_j) v_k(y_j)=0,\]
where we use the notation $[g](y)=g(y^+)-g(y^-)$,
so that the Wronskian $v_1v_2'-v_2v_1'$ still remains continuous across points of discontinuity of $f$ and hence vanishes on $\Rm$. Also, $v_j$ cannot vanish on an open set for otherwise as solutions of a second-order ODE, they vanish between points of discontinuity $y_j$ as well as a across discontinuities from the above jump condition.

This implies $v_2=\alpha v_1$ for $\alpha\not=0$ on the whole line $y\in\Rm$. Indeed, on $(-\infty,y_1)$, $v_1v_2'-v_2v_1'=0$ implies that $v_2=\alpha v_1$ for $\alpha\not=0$ between points where $v_1=0$ since $v_j$ cannot vanish on an open set.  Note that points where $v_1$ vanishes are necessarily isolated for otherwise, using the above jump conditions and structure of the above second-order ODE, we derive that $v_1\equiv0$, which is not possible. By continuity, $v_2'=\alpha v_1'$ where $v_1$ vanishes so that $v_2=\alpha v_1$ remains valid across points where $v_1=0$ as well. The above jump condition implies that $\alpha$ also remains constant across jumps points of $f$.  Indeed, the same relation shows that if $v_2=\alpha v_1$ on the left of a singularity point $y_j$, then the same property holds on the right of it and hence iteratively on the whole line. 

This implies $v_2=\alpha v_1$ and the branch has to be unique.


\paragraph{Case $E(\xi)=-\xi$ revisited.} Since the branches $\xi\mapsto E_n(\xi)$ are analytic, we obtained that a unique branch crossed $E=-\xi$ given by \eqref{eq:Ecrossing}. We then show that $E'(\xi)>-1$ when $E(\xi)=-\xi>0$. This means that this one branch crosses and then converges to $+\infty$ for $\xi\to+\infty$ since no branch is allowed to cross $E=\xi$. Indeed, we obtain from
\[ H\psi=E\psi\]
that
\[ H'\psi + H \psi' = E'\psi + E\psi'\]
and hence that 
\[ (\psi,H'\psi) = E'\]
for a normalized $\|\psi\|=1$. Now $H'={\rm Diag}(\sigma_1,0)$ has eigenvalues $\{-1,0,+1\}$ with $\psi$ proportional to $(1,-1,0)^t$ at almost all $y$ if we want $E'=-1$. This implies $\eta+u=0$ and hence $v=0$, which is not the case when $E(\xi)=-\xi$. Therefore $E'\in (-1,1]$ above. This shows that the branch crossing the line $E(\xi)+\xi=0$ does so with a slope greater than $-1$.

The branch thus converges to $+\infty$ as $\xi\to+\infty$ since it cannot cross the branches $E=\pm\xi$ except at the point given by \eqref{eq:Ecrossing}. It remains in the sector $0<E<-\xi$ for $\xi<\xi_0$ with $\xi_0$ the only point where $E(-\xi_0)=\xi_0$ in \eqref{eq:Ecrossing}.

We have identified a Kelvin branch $E(\xi)=\xi$ and a Yanai branch $E(\xi)$ such that $E(-\xi_0)=\xi_0$ with $E(\xi)\to+\infty$ as $\xi\to\infty$ and $0<E(\xi)<-\xi$ for $\xi<\xi_0$.

\paragraph{Case $|E|\not=|\xi|$.} For the remaining branches, set $\mu=\frac\xi E$ with $|\mu|\not=1$. The first two equations in \eqref{eq:hatHmE} are
\begin{equation}\label{eq:etau}
  \begin{pmatrix}  -E & \xi  \\ \xi & -E \end{pmatrix}  \begin{pmatrix}  \eta \\ u \end{pmatrix} =  - \begin{pmatrix}  D_y v \\ if v \end{pmatrix} ,\qquad 
   \begin{pmatrix}  \eta \\ u \end{pmatrix} = \frac{1}{\xi^2-E^2 }   \begin{pmatrix}  E & \xi  \\ \xi & E \end{pmatrix} \begin{pmatrix}  D_y  \\ if \end{pmatrix} v
\end{equation}
so that 
\[
  (\xi^2-E^2) (D_y \eta  - i f u) = E D_y^2v+\xi D_y(ifv) - if \xi D_y v - if E -f v = E(D_y^2+f^2)v + \xi f' v.
\]
The last equation
\[
  D_y \eta  - i f u = Ev 
\]
is thus equivalent to
\begin{equation}\label{eq:v}
   L_\mu v := (-\partial_y^2+f^2 + \mu f') v = (E^2-\xi^2)v = E^2(1-\mu^2) v.
\end{equation}
We are therefore looking to $v$ solution in $L^2(\Rm)$ of the above equation and such that $D_yv$ and $fv$ are also in $L^2(\Rm)$.
\paragraph{Asymptotics of branches when $f'$ is bounded.}
%

We assume first that $f(y)\in W^{1,\infty}(\Rm)$ is uniformly Lipschitz. We compute
\[
  \fa^*\fa = (-\partial_y+f)(\partial_y+f) = -\partial_y^2 + f^2 - f',\quad \fa\fa^* = -\partial_y^2 + f^2 + f'
\]
so that 
\[
   L_\mu = \fa^*\fa + (1+\mu) f' = \fa\fa^* + (\mu-1) f'.
\]
In particular, we have
\[
   \|\fa v\|^2 + (\mu+1) (f'v,v) = (E^2-\xi^2) \|v\|^2.
\]
This is recast as 
\[
  (\xi^2-E^2)\|v\|^2 = -\frac{\xi+E}{E} (f'v,v) - \|\fa v\|^2 \leq \frac{|\xi+E|}E \|f'\|_\infty  \|v\|^2.
\]
On the sectors $\xi^2>E^2$, we thus obtain that 
\[
 E |\xi-E| \leq \|f'\|_\infty.
\]
When $E<-\xi$, this implies that each such branch $E_n(\xi)\to0$ as $\xi\to-\infty$.

We also observe for the same reason  that 
\[
   \|\fa^* v\|^2 + (\mu-1) (f'v,v) = (E^2-\xi^2) \|v\|^2.
\]
On  $\xi^2>E^2$, we thus obtain that 
\[
  E |\xi+E| \leq \|f'\|_\infty.
\]
When $0<E<\xi$, we thus deduce $E\to0$ as well as $\xi\to\infty$.

\medskip

We are now ready to conclude the proof of Theorem \ref{thm:smooth}.

\begin{proof}[Proof of Theorem \ref{thm:smooth}] We first recall the Kelvin branch $E(\xi)=\xi$. We also identified the Yanai branch satisfying $E(\xi)\to0$ as $\xi\to-\infty$ while $E(\xi)\to\infty$ while $\xi\to\infty$.  The spectral flow \eqref{eq:SF} of each one of these branches therefore equals $1$.  All other branches have a spectral flow equal to $0$. Indeed, they either live in the sector $E(\xi)>|\xi|$ and hence converge to $+\infty$ as $\xi\to\pm\infty$. Otherwise,  they either live in the sector $0<E(\xi)<-\xi$ and then converge to $0$ as $\xi\to-\infty$ as well as to $0$ as $0>\xi\to0$ (or possibly reach $0$ for a negative value of $\xi$), or in the sector $0<E(\xi)<\xi$ and then converge to $0$ as $\xi\to\infty$ and to $0$ as $0>\xi\to0$ (or possibly reach $0$ for a positive value of $\xi$). The sum of the spectral flows of the branches of absolutely continuous spectrum thus equals $2$.
\end{proof}

\paragraph{Special case of $f'$ bounded with $f$ strictly monotone.}  Assume first that $f'>0$. Then no branch $\xi\mapsto E(\xi)$ is possible in the sector $\xi>E>0$. Indeed, from the above equality
$$\|\fa v\|^2 + (\mu+1) (f'v,v) = (E^2-\xi^2) \|v\|^2,$$
we deduce in that case that $\mu>1$ while $E^2-\xi^2<0$ and hence $f'v^2=0$ so that $v=0$. Similarly, we deduce from 
$$\|\fa^* v\|^2 + (\mu-1) (f'v,v) = (E^2-\xi^2) \|v\|^2$$
that no branch $\xi\mapsto E(\xi)$ exists in the sector $-\xi>E>0$ when $f'<0$. 




%
\paragraph{Solution of local jump problem.}
Assume $f(y)$ piecewise-constant with exactly one jump at $y_0$. We denote by $\chi_{y>y_0}$ the function equal to $1$ when $y>y_0$ and $0$ otherwise. Following the appendix in \cite{bal2022topological}, we have the result: 
\begin{lemma}\label{lem:purejump}
 Assume $f=f_+\chi_{y>y_0}+f_-\chi_{y<y_0}$. Define
 \[
   f_o=\frac{f_+-f_-}2,\quad f_e=\frac{f_++f_-}2 ,\qquad \sE= \frac{-f_o\xi}{\sqrt{f_e^2+\xi^2}},\qquad \kappa_\pm =  \sqrt{f_\pm^2+\xi^2-\sE^2} .
 \]
Assume $-f_o\xi>0$.  Then there is a unique normalized solution in $L^2(\Rm;\Cm^3)$ of
 \[
   (\hat H(\xi)-E)\psi=0
 \]
 such that $0<E<|\xi|$. Moreover, $E=\sE$ as given above and $\psi=\varphi$ with
 \[
   \varphi(y) = \begin{pmatrix}  \frac{Ei\kappa_+ + i\xi f_+ }{E^2-\xi^2} \\ \frac{\xi i\kappa_+ + iE f_+ }{E^2-\xi^2} \\ 1\end{pmatrix} v(y_0)  e^{-\kappa_+(y-y_0)}  \chi_{y>y_0} \ + \ 
    \begin{pmatrix}  \frac{Ei\kappa_+ + i\xi f_+ }{E^2-\xi^2} \\ \frac{-\xi i\kappa_- + iE f_- }{E^2-\xi^2} \\ 1\end{pmatrix} v(y_0) e^{\kappa_-(y-y_0)} \chi_{y<y_0},
 \]
 where $v(y_0)\in\Cm$ is a normalizing constant.
\end{lemma}
Note that $\varphi(y,-\xi)=\bar\varphi(y,\xi)$ reflecting the real-valuedness of solutions to \eqref{eq:H}.
\begin{proof}
To make sense of the differential equations and obtain solutions in the domain of definition of $\hat H$, we impose the continuity of $\eta$ and $v$ at $y=y_0$. Note that no such condition may be imposed on $u$ a priori. We then obtain the jump conditions at $y=y_0$:
\[
  iv(y_0)[f] = E [u] ,\quad \xi v(y_0) [f] = E [\partial_y v],\quad [\partial_y\eta]= -[fu].
\]
We recall the notation $[g](y)=g(y^+)-g(y^-)$.
This implies that while $\eta$ and $v$ are continuous at $y=y_0$, then $u$ is not in light of the first relation.


Since $|E|<|\xi|$, we can eliminate $(\eta,u)$ as in \eqref{eq:etau} and verify that we are looking for a solution of the form
\[
  \varphi(y) = \begin{pmatrix} \eta \\ u_- \\ v \end{pmatrix} e^{\kappa_-(y-y_0)}  \chi_{y<y_0} + \begin{pmatrix} \eta \\ u_+ \\ v \end{pmatrix} e^{-\kappa_+(y-y_0)}  \chi_{y>y_0} .
\]
Thus
\[
   [ \partial_y v] + (\kappa_++\kappa_-) v =0,\quad [ \partial_y \eta] + (\kappa_++\kappa_-)  \eta =0.
\]
From the equation \eqref{eq:v} for $v$, we deduce 
\[
  E^2-\xi^2 = f_+^2-\kappa_+^2= f_-^2-\kappa_-^2, \qquad \kappa_++\kappa_-+\frac\xi E(f_+-f_-)=0.
\]
This implies that $\kappa_\pm>0$ are real-valued and $\xi E f_o<0$. From $\kappa_+^2-\kappa_-^2=f_+^2-f_-^2$, we deduce that
$\kappa_+-\kappa_-  + \frac E\xi (f_++f_-)=0$.  Let us define $\nu=\frac E \xi$. Then,
\[
  \mu_+ + \nu^{-1} f_o + \nu f_e=0 \quad \mbox{ so that } \quad \mu_+^2 = \nu^{-2} f_o^2 + f_of_e + \nu^{2} f_e^2 = f_+^2+\xi^2(1-\nu^{2}).
\]
This gives the equation for $\nu$, or equivalently $E$, using $f_+^2-2f_of_e=f_o^2+f_e^2$:
\[
  (f_e^2+\xi^2)\nu^4 - (f_o^2+f_e^2+\xi^2)\nu^2 + f_o^2=0 \ \mbox{ or } \ (f_e^2+\xi^2) (\nu^2-1)(\nu^2-\frac{f_o^2}{f_e^2+\xi^2}) =0.
\]
Looking for solutions $0<E<|\xi|$ with $\xi E f_o<0$, the only admissible solution is
\begin{equation}\label{eq:Exi}
  E = \frac{-\xi f_o}{\sqrt{f_e^2+\xi^2}}.
\end{equation}
Solving for $(\eta,u)$ in \eqref{eq:etau} then gives the result stated in the lemma.
\end{proof}

\paragraph{Asymptotics of branches in the presence of a finite number of jumps.}
We assume that $f$ has a finite number of jumps at points $y_j$ for $1\leq j\leq J$ and recall  the assumption that for some $\eta>0$, then $f(y)=f(y_j^-)$ on $(y_j^--2\eta,y_j^-)$ and $f(y)=f(y_j^+)$ on $(y_j^+, 2\eta+y_j^+)$. We assume that each jump $|[f](y_j)|\geq\eta>0$ and $f$ is sufficiently smooth between jumps. As a consequence, from ODE theory, we obtain that $\psi(y)$ is correspondingly smooth between jumps. 

At any jump point $y_j$ of $f(y)$, let us define $\varphi_j(y)$ as the solution given by Lemma \ref{lem:purejump} where $(y_0,f_+,f_-)$ is replaced by $(y_j,f(y_j^+),f(y_j^-))$. 
We denote by $0<\sE_j(\xi)$ the corresponding eigenvalue. 
Let us further define the cut-off function $\phi(y)\in C^\infty_c(\Rm)$ such that $\phi(y)=0$ for $|y|\geq2$ while $\phi(y)=1$ for $|y|\leq1$. We then observe by standard estimates that 
\[
 (\hat H(\xi)-\sE_j) \varphi_{j\eta}(y) = S(\xi),\quad  \varphi_{j\eta}(y) = \varphi_j(y) \phi(\frac{y-y_j}\eta),\quad\|S(\xi)\|_2 \leq e^{-C_\eta|\xi|},
\]
for some constant $C_\eta>0$. We thus obtain the following result:
\begin{corollary}\label{cor:jumps}
  Let $f(y)$ be as above. Then, for all $|\xi|>\xi_0(\eta)$ and $1\leq j\leq J$, there exist $(\psi_j,E_j)$ solutions of
  \[
    (\hat H(\xi)-E_j)\psi_j=0
  \]
  with $|E_j-\sE_j|\leq C e^{-C_\eta|\xi|}$ and $\|\psi_j - \varphi_{j\eta}  \|\leq C e^{-C_\eta|\xi|}$ for a normalized $\|\psi_j\|=1$.
\end{corollary}
\begin{proof}
 This is coming from the fact that $\hat H(\xi)$ is self-adjoint so that there is an element in the spectrum of $\hat H(\xi)$ close to $\sE_j$ for $|\xi|$ sufficiently large. This is a direct consequence of the spectral theorem \cite{Kato66}. Indeed, the above construction of $\varphi_{j\eta}$ implies that if $\sE_j$ is not an eigenvalue of $\hat H(\xi)$ (if it is we are done), then $\sup_{E\in\sigma(\hat H(\xi))} |E-\sE_j|^{-1} = \|(\hat H(\xi)-\sE_j)^{-1}\|\geq e^{C_\eta|\xi|}$. The eigenvector is then close to the quasi-eigenvector as stipulated in the corollary.
\end{proof}
\begin{lemma}\label{lem:smallinner}
  Let $\psi$ be a solution of the problem 
  \[
    (\hat H(\xi)-E)\psi=0
  \]
  for $0<E<|\xi|$ with $\|\psi\|=1$. Let $v(y_j)$ be the third component of $\psi(y_j)$ and $v_{j}(y_j)$ the third component of $\psi_{j}(y_j)$ defined in Corollary \ref{cor:jumps}. Then
  \[
     \big| (\psi, \psi_{j}) \big| \geq \frac{|v(y_j)|}{|v_{j} (y_j)|} -C e^{-C_\eta |\xi|},
  \]
  for some constant $C>0$ independent of $j$ and $|\xi|$ sufficiently large.
\end{lemma}
\begin{proof}
Since $\varphi_{j\eta}$ is supported in $(y_j-2\eta,y_j+2\eta)$, then $\psi(y)$ restricted to that interval is also given by the solution of Lemma \ref{lem:purejump}. Since $\psi$ and $\varphi_{j\eta}$ are normalized, $v(y_j)$ cannot be significantly larger than $v_{j\eta}(y_j)$. When these two quantities are equal, then we have $| (\psi, \varphi_{j\eta})| \geq 1 -C e^{-C_\eta |\xi|}$. Therefore, when $v(y_j)=\lambda v_{j\eta}(y_j)$ for $0\leq\lambda\leq1$, we have $| (\psi, \varphi_{j\eta})| \geq \lambda -C e^{-C_\eta |\xi|}$. The above estimate then follows from corollary \ref{cor:jumps} replacing $\varphi_{j\eta}$ by $\psi_j$. 
\end{proof}



\begin{theorem} \label{thm:jumps}
  Let $f$ be as above with a finite number $J$ of jumps. Then in the sector $0<E(\xi)<|\xi|$, there are exactly $J$ branches of absolutely continuous spectrum of the equation $(\hat H(\xi)-E(\xi)) \psi=0$ converging to non-vanishing limits 
\[
     E_j(\xi) \to | f_{jo}| \ \mbox{ as } \ \left\{  \begin{array}{ccl} \xi\to-\infty & \mbox{ when } & f_{jo}>0,  \\ \xi\to \infty & \mbox{ when } & f_{jo}<0. \end{array} \right.
\]
All other branches of the equation converge to $0$ as $|\xi|\to\infty$.  
\end{theorem}
\begin{proof}
  The result on the convergence of the branches is clear from Corollary \ref{cor:jumps}. 
  
  From the variational formulation of the equation for $v_j(y)$, we obtain that
  \[
   \|\fa v\|^2 + (\mu+1) (f'v,v) + \sum_{j=1}^J (\mu+1) [f]_j |v(y_j)|^2= (E^2-\xi^2) \|v\|^2,
  \]
  where $(f'v,v)$ involves the bounded part $f'$ of the derivative of $f$ and $[f]_j=2 f_{jo}$. When $\xi<0$, we obtain as earlier that
  \[
  (\xi^2-E^2)\|v\|^2  \leq \frac{|\xi+E|}E \Big( \|f'\|_\infty  \|v\|^2 + \sum_{j=1}^J [f]_j |v(y_j)|^2\Big) .
\]
From Lemma \ref{lem:smallinner}, we obtain that $|v(y_j)|\leq C |\xi|^{\frac12} e^{-C_\eta |\xi|}\|v\|$ from the orthogonality $(\psi,\psi_j)=0$ and the fact that $\|v_j\|$ is of order $O(1)$ uniformly in $|\xi|$ sufficiently large. This implies again that $E(\xi)\to0$ as $\xi\to-\infty$. The variational formulation involving $\|\fa^*v\|$ instead provides the convergence when $\xi\to\infty$.
\end{proof}

\medskip

We can finally obtain the proof of our main theorem.
\begin{proof}[Proof of Theorem \ref{thm:SF}]
 We know that $2\pi\sigma_I$ is given by the spectral flow of $\hat H(\xi)$. 
 As a corollary of Theorem \ref{thm:jumps}, all branches associated to jumps converge to an element in $\mE_L$ as $\xi\to-\infty$ or to an element in $\mE_R$ as $\xi\to+\infty$. These branches have to converge to $E=0$ as $\xi\to0$ (or reach $0$ at a point $x_0$) as they remain in the sectors $0<E<|\xi|$. The spectral flow of these branches is then equal to $-1$ for $j$ in $\mJ_L(E)$ while it is given by $+1$ for $j\in \mJ_R(E)$.
 
 The branch $E(\xi)=\xi$ always contributes $1$ to the spectral flow. We also know that branches do not cross as they are simple. As a consequence, the only branch crossing the half-line $0<E=-\xi$ converges to $+\infty$ as $\xi\to\infty$ and converges to either the largest element in $\mE_L$ or to $0$ as $\xi\to-\infty$ when $\mE_L$ is empty. This Yanai branch has a spectral flow equal to $1$ when $\alpha=E$ is larger than the latter limit and $0$ otherwise.
\end{proof}

\paragraph{Discreteness of the spectrum of $\hat H(\xi)$.}  We finally prove the following result:

\begin{lemma}\label{lem:discrete} Assume that $f(y)$ is a locally bounded function with unbounded range. We also assume that $f(y)$ is sufficiently smooth between a finite number of point $y_j$ where it may be discontinuous.

Then $\hat H(\xi)$ has purely discrete spectrum away from $E=0$.

For $\xi=0$, then $\hat H(0)$ has essential (point) spectrum at $0$. 

For $\xi\not=0$, then $\hat H(\xi)$ has essential spectrum at $0$ only when the operator of multiplication by $f'(y)$ has essential spectrum (i.e., $f'(y)=0$ on an open set).

This allows us to define branches of absolutely continuous spectrum of $H$. Moreover, the branches $\xi\mapsto E_n(\xi)$ for $n\in\Zm$ are real analytic.
\end{lemma}
\begin{proof}
We assume that $f(y)$ has unbounded range. This implies that the operator $(-\partial^2_y+f^2(y) + \mu f'(y))$ has compact resolvent \cite{teschl2009mathematical}. The discreteness of the spectrum of $\hat H(\xi)$ for $\xi\not=0$ depends on this property and the Weyl criterion for essential spectrum.

Let $\xi\in\Rm$ be fixed and assume that $|\lambda|\not=|\xi|$ for $\lambda\in\Rm\backslash\{0\}$. Assume that $\lambda \in \sigma_{\rm ess}(\hat H(\xi))$. Then there is a Weyl singular sequence $\psi_k$ such that $\|\psi_k\|=1$, $(\hat H-\lambda)\psi_k\to0$ in the $L^2$ topology, and the $\psi_k$ have no convergent subsequence (and hence can be chosen orthonormal).

Define $(\hat H-\lambda)\psi_k=g_k$. Since $|\lambda|\not=|\xi|$, we can eliminate $(\eta,u)$ as performed earlier in \eqref{eq:etau} and get that for $\mu=\xi/E$,
\[
  L_\xi v_k := (-\partial^2_y + f^2(y) + \mu f'(y) + \xi^2-\lambda^2) v_k = s_k
\]
where $s_k$ involves the source terms $g_k$ possibly multiplied by $f(y)$ or differentiated once. There cannot be an accumulation point for $v_k$ for otherwise $(\eta_k,u_k)$ would also converge along that subsequence and the original $\psi_k$ would then have an accumulation point, which is impossible.

And yet, from the above equation, the compactness of the resolvent operator of the Sturm-Liouville operator $L_\xi$ \cite{teschl2009mathematical} and the compactness of that resolvent operator composed with $\partial_y$ or with multiplication by $f(y)$, as one easily verifies by elliptic regularity, implies that $v_k$ belongs to a compact subset of $L^2(\Rm)$ and hence has an accumulation point. This implies that $\hat H(\xi)$ only has discrete spectrum away from $\pm \xi$.

Let us now assume that $\xi\not=0$ and that $\lambda=\xi$. This implies $i\fa^* v_k$ converges to $0$ so that $v_k$ converges to $0$. We also observe that $\eta_k-u_k$ converges to $0$ and that $\eta_k$ and $u_k$ then converge to the unique normalized solution of $\fa u=0$. Again, this implies that $\lambda=\xi$ is not in the essential spectrum of $\hat H$. When $\lambda=-\xi$, we similarly obtain that $v_k$ converges to the solution of $\fa v=0$. This implies that $\eta_k+u_k$ converges as well when $\xi\not=0$ while $\eta_k+u_k$ has to converge to $0$.

It thus remains to analyze the case $\lambda=0$. When $\xi=\lambda=0$, we observe that $\hat H(0)$ does have essential spectrum at $\lambda=0$. We already saw that $\hat H(\xi)$ had discrete spectrum away from $\lambda=0$. When $\lambda=0$, we observe that $v_k\to0$. Any solution of $\partial_y \eta_k+ f(y) u_k=0$ generates a singular Weyl sequence, and this space of solutions is clearly infinite.

When $\lambda=0$ and $\xi\not=0$, we observe that the elimination of $(\eta,u)$ yields $f'(y)v_k\to0$. When $f'(y)=0$ on an open set, then there is an infinite number of orthogonal solutions to the equation $f'(y)v_k=0$ and hence the presence of singular (pure point) spectrum. Alternatively, when $f'(y)$ vanishes only on a discrete set of points, then $v_k\to0$ and hence so do $\eta_k$ and $u_k$ when $\xi\not=0$. In this case, $\hat H(\xi)$ has no essential spectrum (at $\lambda=0$ or anywhere else). For the same reason, it also has no point spectrum at $\lambda=0$ in the latter case since $\hat H(\xi)\psi=0$ with $\|\psi\|=1$ implies $\psi=0$.

The above results allow us to define the branches $\xi\mapsto E_n(\xi)$ for $n\in\Zm$ starting with the branch $E_0(\xi)=\xi$, say. Since $\xi\mapsto\hat H(\xi)$ is clearly real analytic and the branches are one-dimensional, we deduce from \cite[Chapter VII.1.1]{Kato66} that away from $E=0$, the branches are indeed real analytic. We already saw that by simplicity, the branches could not cross either except possibly at $E=\xi=0$, where they do cross.
\end{proof}
%
\section{Numerical simulations} \label{sec:num}

We now present an algorithm and its implementation allowing us to approximate the branches $E_n(\xi)$. This requires truncating the computational domain to a support $y\in[-L,L]$. We then impose periodic boundary conditions, in other words solve each problem on a one-dimensional torus (circle). We also impose that $f(y)$ is appropriately modified so that it does not generate any jump in the vicinity of $y=L$. The resulting ordinary differential equation is then discretized by a standard finite difference scheme. In the simulations presented in this paper, we chose $L=11$ while the number of discretization points is equal to $m=5001$.

We face two main difficulties in discretizing $\hat H(\xi)$ in \eqref{eq:hatH}. The first one is that the periodization of the domain and hence of the profile $f(y)$ creates a domain wall in the vicinity of $y=L$ with opposite direction to that of the domain wall created near $y=0$. This form of Fermion doubling is well known \cite{delplace2017topological,zhu2023topology}. The spectrum of the periodized operator is therefore totally different from that of the initial operator $\hat H(\xi)$. A second difficulty is that spurious numerical modes with rapid oscillations, e.g., with wavenumber comparable to $m$, inevitably appear in the spectrum of the discrete operator. Such high wavenumber spurious numerical modes typically correspond to large eigenvalues. However, as seen in Fig. \ref{fig:1}, $\hat H(\xi)$ has both large and small eigenvalues that are not well separated spectrally from the spurious numerical modes. 

\begin{figure}[ht!]
\begin{center}
  \includegraphics[width=7.5cm]{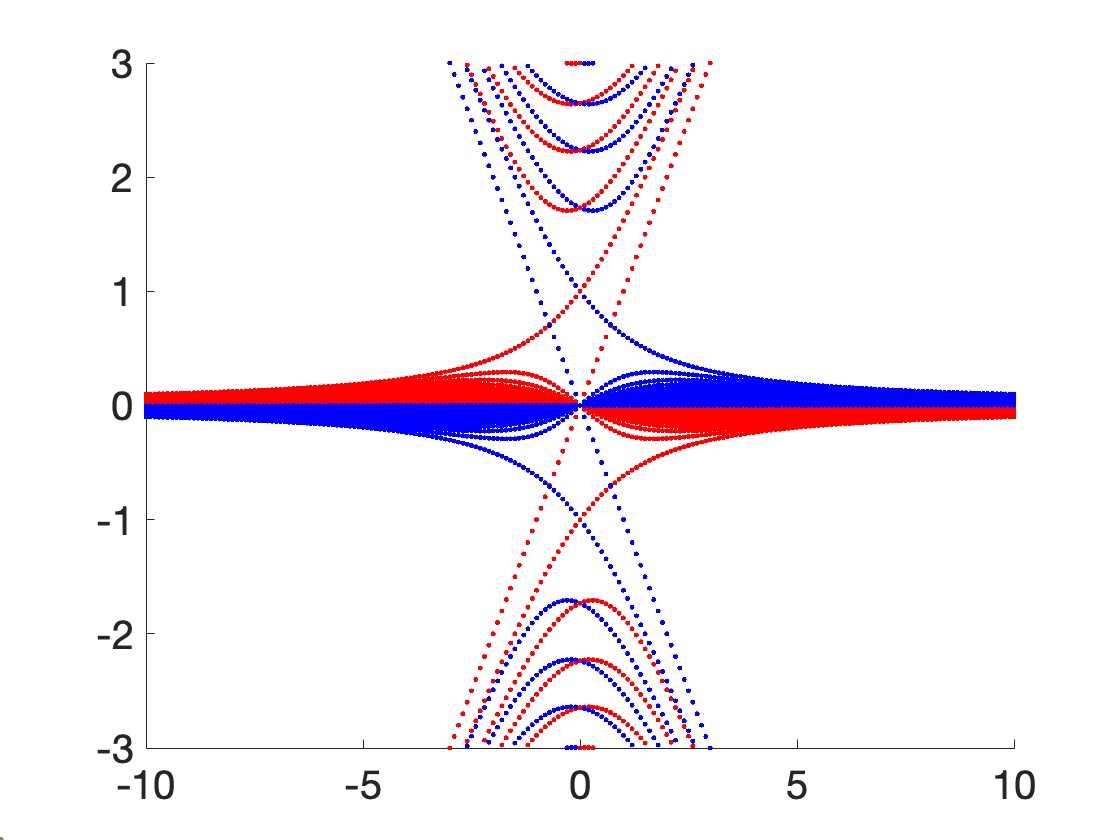} \hspace{.1cm} \includegraphics[width=7.5cm]{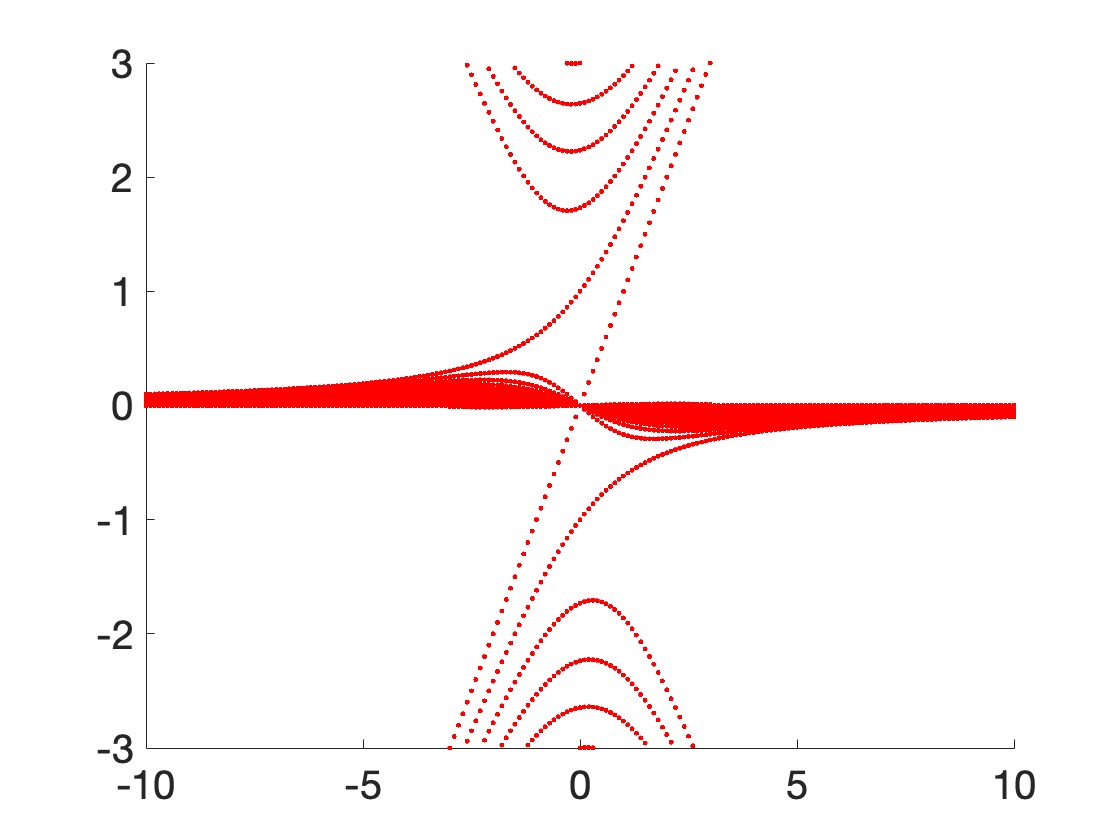}
\end{center}
\caption{Numerical simulation of the spectrum of $H$ for $f(y)=y$. Left: full numerical spectrum without elimination of eigenvalues. Right: spectrum after elimination of eigenvectors concentrating near the domain wall $y=L$. 
}
\label{fig:2}
\end{figure}



These two problems are remedied by analyzing all numerically computed eigenvectors $\psi(y)$. The eigenvectors concentrating more in the vicinity of $y=L$ than they do in the vicinity of $y=0$ are eliminated.  In practice, we keep eigenvectors when the $l^2$ norm squared of the discrete vector on $(-\frac9{10}L,\frac{9}{10}L)$ is more than half the $l^2$ norm squared on the whole interval $(-L,L)$. Spurious numerical modes are also eliminated when their discrete Fourier transform satisfies that the $l^2$ norm squared on one fifth of the frequency range is at least $9/10$ of the whole $l^2$ norm squared.

Fig. \ref{fig:2} displays the results of numerical simulations for the profile $f(y)=y$. The left panel shows the numerical spectrum of the periodized operator, which becomes approximately (but not exactly) symmetrical against the transformation $\xi\to-\xi$. The spectral flow of the periodized operator clearly vanishes. The right panel displays the spectrum after filtering of the spurious modes generated by the periodization. 

\begin{figure}[ht!]
\begin{center}
  \includegraphics[width=7.7cm]{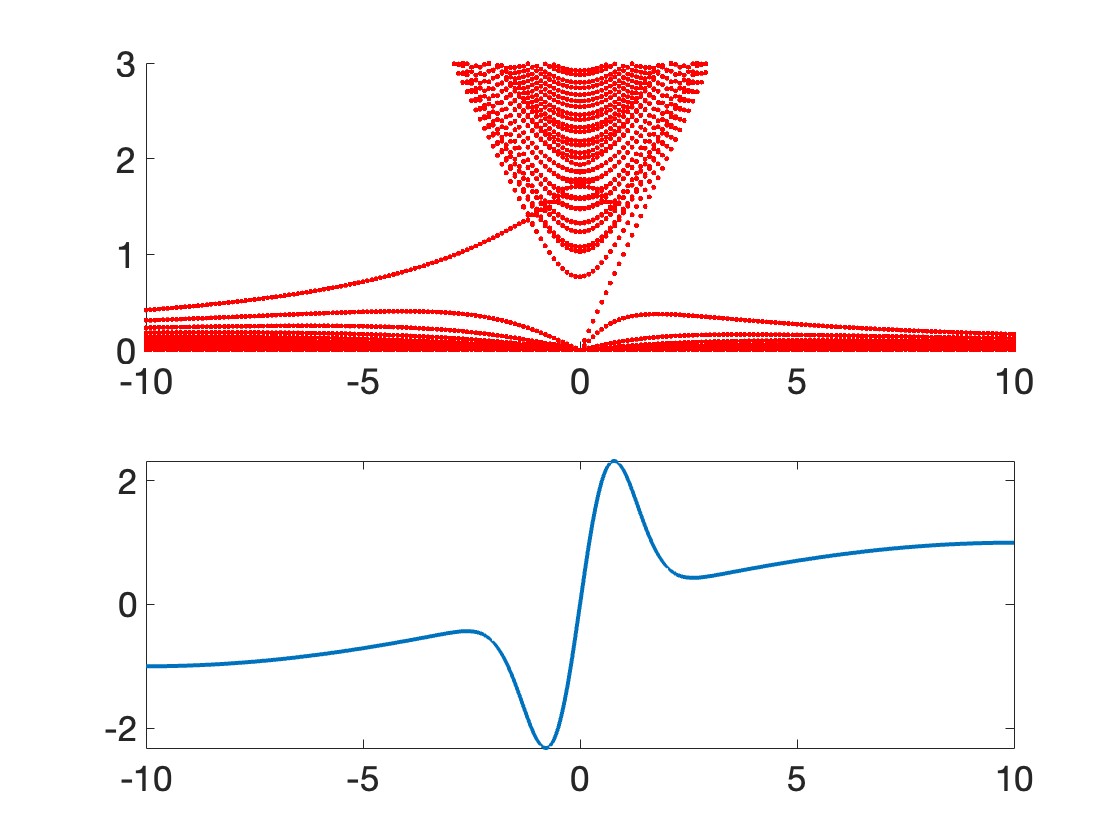}  
  \includegraphics[width=7.7cm]{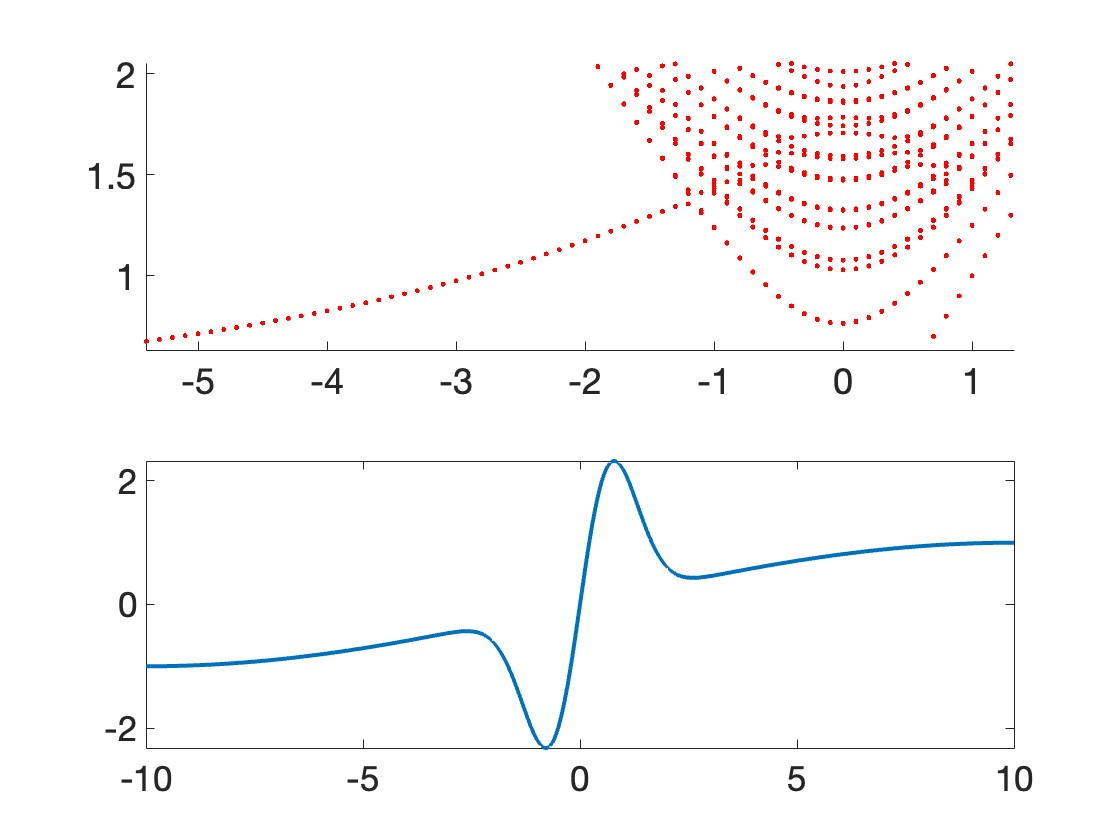}
\end{center}
  \caption{Left: spectrum of $H$ when $f$ is the profile displayed in the bottom panels. The branches converge to $0$ as $|\xi|\to0$ as expected from Theorem \ref{thm:smooth}. The zoomed-in right panel shows that even though branches almost touch, they do not cross as expected from their theoretical simplicity.}
\label{fig:3}
\end{figure}
Figure \ref{fig:3} shows the (filtered) spectrum of a smooth Coriolis parameter displaying large oscillations. As expected from the theory, the spectral flow of such an operator equals $2$ for any $\alpha\not=0$.

\begin{figure}[ht!]
\begin{center}
  \includegraphics[width=7.7cm]{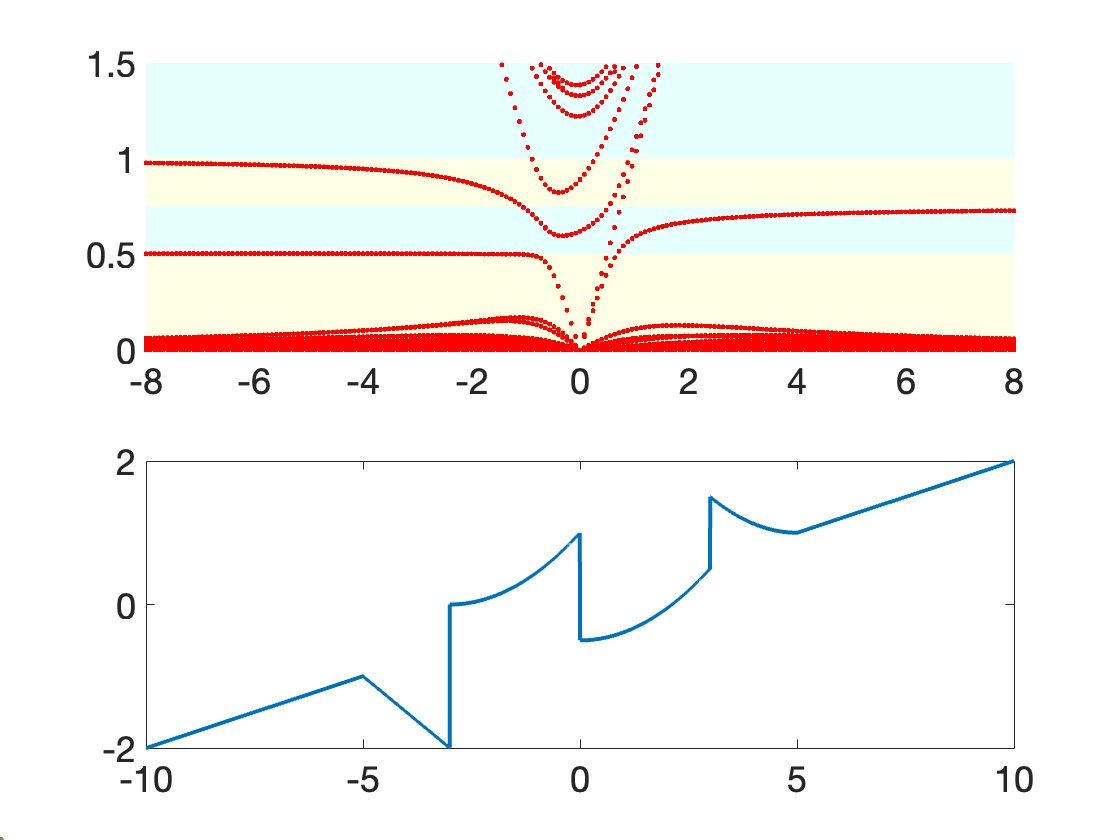}\hspace{.1cm}
  \includegraphics[width = 7.7cm]{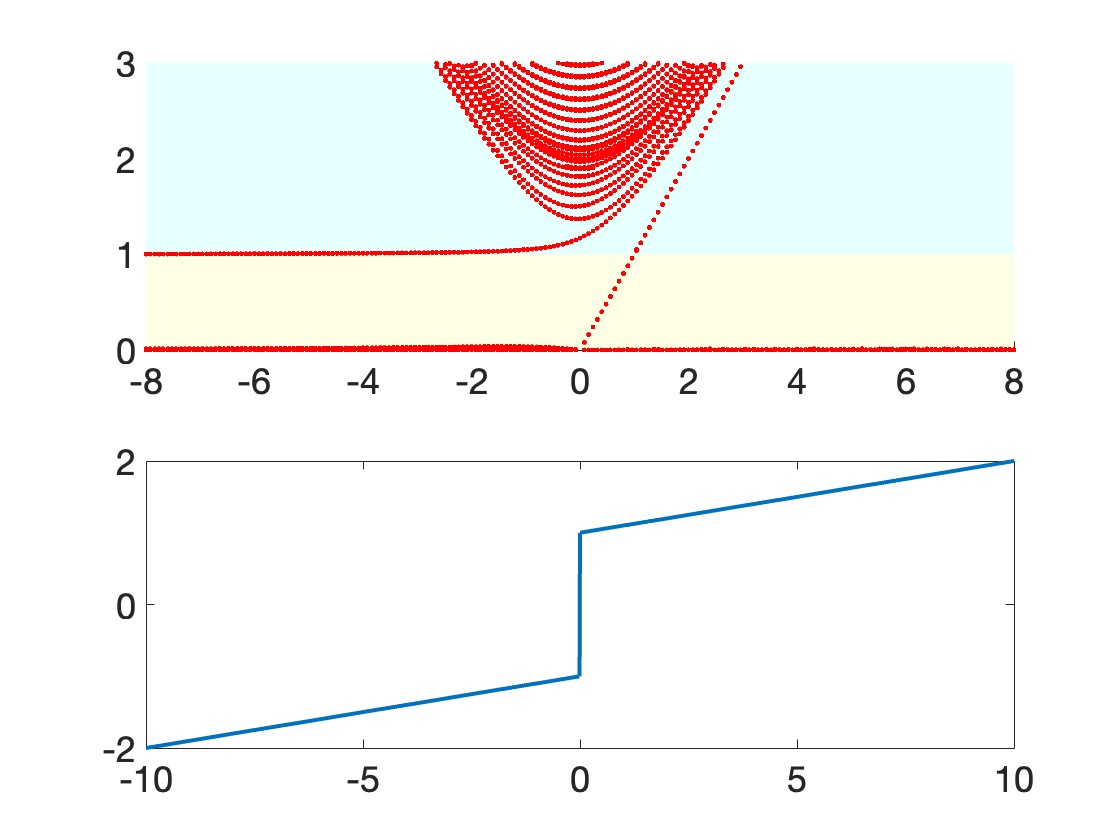}
\end{center} 
\caption{Left: profile with two positive jumps $+2$ and $+1$ and one negative jump $-3/2$.  Right: profile with one positive jump of size $+2$.
The filtered numerical spectra are consistent with the results of in Theorem \ref{thm:SF}.}
\label{fig:4}
\end{figure}


As an illustration of the violation of the BEC, we consider the profile displayed on the left in Fig. \ref{fig:4}. In that setting, we observe that $2\pi\sigma_I(\alpha)$ is given by $1$ for $0<\alpha<1/2$ or $3/4<\alpha<1$  (i.e.,  $\varphi'$ supported in the yellow region in Fig. \ref{fig:4}) and by $2$ for $1/2<\alpha<3/4$ or $\alpha>1$ ($\varphi'$ supported in the blue region). Only in the latter case is the BEC satisfied.

The simulation on the right of Fig. \ref{fig:4} shows a spectral flow equal to that in Fig. \ref{fig:1}, for a profile linear profile, except for the presence of one jump. The spectral flow equals $1$ for $0<\alpha<1$ and equals $2$ for $\alpha>1$. Note the absence of spectrum in the sectors $0<E<\xi$ as expected from theory since $f'(y)>0$ for $y$ away from points of discontinuity.

%
\begin{figure}[ht!]
\begin{center}
        \includegraphics[width = 7.5cm]{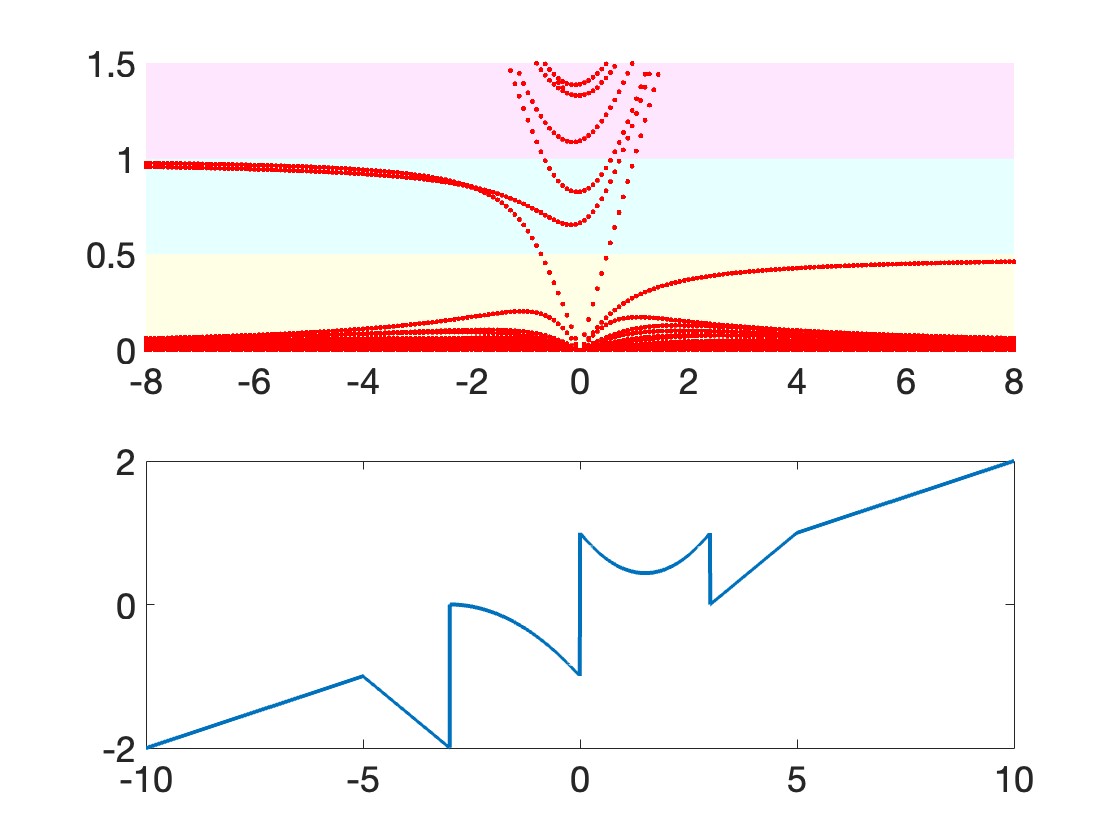}
     \includegraphics[width = 7.5cm]{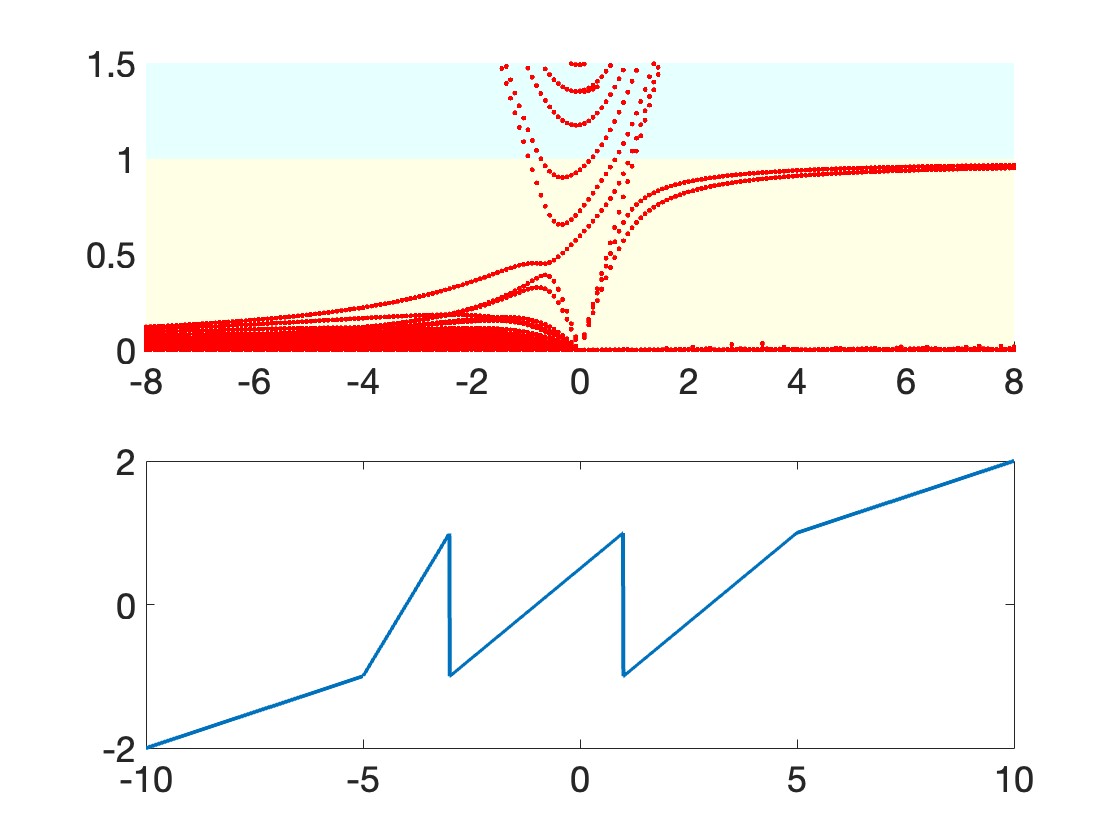}
\end{center}
\caption{Left: profile with two positive jumps of size $+2$ and one negative jump $-1$.  Right: profile with two negative jumps of size $-2$.
The filtered numerical spectra are consistent with the results of in Theorem \ref{thm:SF}.}
\label{fig:5}
\end{figure}

Two profiles involving multiple branches converging to the same value as $|\xi|\to\infty$ are presented in Fig. \ref{fig:5}. Only for frequencies $\alpha>1$ is the bulk-edge correspondence valid. We observe on the left panel a Yanai branch converging to $E=1$ as $\xi\to-\infty$ while a second branch also converges to $E=1$ in that same limit. On the right panel, we also observe two separate branches both converging to $E=1$ as $\xi\to\infty$, corresponding to a vanishing spectral flow for $0\alpha<1$ while the value of $2$ predicted by the BEC holds when $\alpha>1$.

\section*{Acknowledgment} This work was supported in part by the US National Science Foundation under grant DMS-2306411.

{\small

}

\end{document}